\documentclass[12pt]{amsart}
\usepackage[margin=1in]{geometry}
\usepackage{amsthm,amsmath,times,amsfonts,amssymb}
\usepackage{dyck_paths}

\newtheorem{theorem}{Theorem}

\newtheorem{prop}[theorem]{Proposition}
\newtheorem{example}[theorem]{Example}

\newtheorem{cor}[theorem]{Corollary}

\theoremstyle{remark}
\newtheorem{remark}[theorem]{Remark}

\newcommand\Area{ \mathsf{Area} }
\newcommand\Dinv{ \mathsf{Dinv} }
\newcommand\Skips{ \mathsf{Skip} }
\newcommand\Inv{ \mathsf{Inv} }
\newcommand\area{ \mathsf{area} }
\newcommand\dinv{ \mathsf{dinv} }
\newcommand\skips{ \mathsf{skip} }
\newcommand\inv{ \mathsf{inv^m} }
\newcommand\ides{ \mathsf{iDes} }
\newcommand\des{ \mathsf{Des} }
\newcommand\leg{ \mathsf{leg} }
\newcommand\arm{ \mathsf{arm} }
\newcommand\bounce{ \mathsf{bounce} }
\newcommand\tdinv{ \mathsf{tdinv} }
\newcommand\maxtdinv{ \mathsf{maxtdinv} }

\def \Z {\mathbb Z}
\def \R {\mathbb R}
\def \Q {\mathbb Q}

\def \PF {\mathsf{PF}}

\renewcommand{\th}{^{\text{th}}}

\title{A Rational Catalan Formula for $(m,3)$-Hikita Polynomials}
\author{Ryan Kaliszewski and Debdut Karmakar}

\begin{document}
\maketitle
\begin{abstract} Building upon a recent formula for $(3,m)$-Catalan polynomials, we describe a formula for $(3,m)$-Hikita polynomials in terms related to Catalan polynomials.  This formula shows a surprising relation among coefficients of Hikita polynomials and implies deeper recursive relations and proves the $q,t$-symmetry of $(3,m)$-Hikita polynomials.

\end{abstract}

\section{Introduction}

In the early 1990's Garsia and Haiman introduced an important sum of
functions in
$\Q(q,t)$, the $q,t$-Catalan polynomial $C_n(q,t)$,
which has since been shown to have interpretations in terms of algebraic
geometry and representation theory.
These classical $q,t$-Catalan polynomials are given by
\begin{align} C_n(q,t) & =\sum_{\pi}q^{\dinv(\pi)}t^{\area(\pi)} \\  \label{pt2}  &=\sum_{\pi}q^{\area(\pi)}t^{\bounce(\pi)},\end{align}
where the sums are over all Dyck paths $\pi$ from $(0,0)$ to $(n,n)$, with \eqref{pt2} due to \cite{Haglund03}.  For an overview of the classical $q,t$-Catalan polynomials and Dyck paths, see \cite{Garsia96, Garsia02, Haglund05,Haglund08}.

Recently, a valuable generalization of the classic $q,t$-Catalan polynomial has come to light \cite{Hikita12}.  For positive integers $m,n$ that are coprime, these \emph{$(m,n)$-rational} $q,t$-Catalan polynomials have a similar description to the classic case
\begin{equation} C_{m,n}(q,t) =\sum_{\pi}q^{\dinv(\pi)}t^{\area(\pi)}, \end{equation}
where the sum is over all \emph{rational} Dyck paths $\pi$ from $(0,0)$ to $(m,n)$.

A rational Dyck path is a path in the $m\times n$ lattice which proceeds by north and east steps from $(0,0)$ to $(m,n)$ and which always remains above the main diagonal $y=\frac{m}{n}x$.  The collection of cells lying above a Dyck path $\pi$ always forms an English Ferrers diagram $\lambda(\pi).$  In 2015, Kaliszewski and Li gave an explicit formula of $(3,n)$-rational $q, t$-Catalan polynomials \cite{Kalis15,Kalis16} for $n$ not divisible by 3,
\begin{equation} \label{klformula} C_{3,n}(q,t)=\sum_{0\leq i <n/3} s_{(n-1-2i,i)}(q,t). \end{equation}
Since $C_{m,n}(q,t)=C_{n,m}(q,t)$, by reflecting the Dyck paths about the line $y=x$, this gives an explicit formula for $C_{n,3}$ as well.

In \cite{Hikita12}, Hikita extended parking functions and their statistics to the $(m, n)$-rational case and defined polynomials 
\begin{equation} \label{hikitapoly} \mathcal{H}_{m,n}(X;q,t)=\sum_{\PF}t^{\area(\PF)}q^{\dinv(\PF)}F_{\ides(\PF)}(X), \end{equation}
where the sum is over all parking functions $\PF$ over all $(m,n)$-Dyck paths and the $F$ are the (Gessel) fundamental quasisymmetric functions indexed by the inverse descent set of the reading word of $\PF$. When $m=2$ or $n=2$, the structure of Hikita's polynomials  have been completely described by \cite{Leven14}.  

Since, for each $(m,n)$-rational Dyck path there is a parking function on that path with the same statistics as the underlying Dyck path, the $(m,n)$-rational $q,t$-Catalan polynomial appears within $\mathcal H_{m,n}(X;q,t)$.  Specifically, $C_{m,n}(q,t)$ is the coefficient of $F_{\emptyset}$ in the expansion given in Equation \eqref{hikitapoly}.

In this paper we explore a relation between the coefficients of the two- and three-row Hikita polynomials $\mathcal{H}_{m,2}$ and $\mathcal{H}_{m,3}$ and the two- and three-row Catalan polynomials $C_{m,2}$ and $C_{m,3}$.  We will prove formulas:
\[ \mathcal{H}_{m,2}(X;q,t)=C_{m,2}(q,t)s_{(1,1)}(X)+C_{m-2,2}(q,t)s_{(2)}(X) \]
and
\[ \mathcal{H}_{m,3}(X;q,t)=C_{m,3}(q,t)s_{(1,1,1)}(X)+\left(K_{m-1,3}(q,t)+K_{m-2,3}(q,t)\right)s_{(2,1)}(X)+ C_{m-3,3}(q,t)s_{3}(X), \]
where 
\[ K_{m,3}(q,t)=\sum_{0\leq i<m/3} s_{(n-1-2i,i)}(q,t). \]
Note that for $m$ not divisible by 3,
 \[ K_{m,3}(q,t)=C_{m,3}(q,t), \]
by Equation \eqref{klformula}.

\subsection{Rational Dyck Paths}

Suppose that $m,n$ are positive coprime integers.  Construct the \emph{$(m,n)$-lattice} by drawing a rectangular integer lattice in $\R^2$ whose southwest corner lies on the origin and whose northeast corner lies on the point $(m,n)$.  The cell whose northeast corner lies on the point $(u,v)$ will be referred to as cell $(u,v)$.  Thus we can define the $i\th$ row as the set of cells
\[ \{(u,i)|1\leq u\leq n\} \]
and the $j\th$ column as the set of cells
\[ \{(j,v)|1\leq v\leq m\}. \]
The \emph{$(m,n)$-diagram} is the $(m,n)$-lattice where each cell $(u,v)$ contains an integer $a$ that satisfies
\begin{equation} a = mn-un-(n+1-v)m. \end{equation}
We call $a$ the \emph{rank} of cell $(u,v)$, denoted by $\gamma_{m,n}(u,v)$ or simply $\gamma(u,v)$ when the parameters are clear from context.  Since $m$ and $n$ are coprime, there are no duplicate ranks.

An \emph{$(m,n)$-rational Dyck path} or $(m,n)$-Dyck path is path on the $(m,n)$-diagram that begins at $(0,0)$ and ends at $(m,n)$.  The path can only consist of northward and eastward steps and must always lie above the diagonal $y=\frac{n}{m}x$.

An $(m,n)$-Dyck path partitions the cells within the $(m,n)$-diagram into two sets.  Since the path must lie above the diagonal, one of the sets will always contain the southeast corner of the $(m,n)$-diagram.  We say that any of the cells in this set are \emph{below the path}.  The other set of cells, which may be empty, are \emph{above the path}.  We say that a cell is \emph{on the path} if its western edge is part of the path.
We will say that a rank is above, below, or on the path if the cell containing it is above, below, or on the path; respectively.

\begin{example} A $(4,7)$-Dyck path:
\[ \pi=\dyckpath{4}{7}{0,0,0,0,0,2,2}{3cm}\;.\]
The cells that are above the path are colored gray and those that are below the path are colored cyan.  Rank 6 is above the path, while rank 5 is below the path.  The ranks that are on the path are $-7,-3,-1,1,3,5,$ and $9$.

Here $\gamma(1,5)=9$ and $\gamma(4,3)=-20$.

\end{example}

The set of cells above the path always has the shape of a Ferrers diagram (in English notation).  This is because the only allowed moves are northward steps and eastward steps, so the number of cells above the path in each row must be weakly increasing from bottom to top.

\subsection{Statistics on Rational Dyck Paths}

Let $\pi$ be an $(m,n)$-Dyck path for coprime $m$ and $n$.  We partition the cells containing positive ranks into three sets, $\Area(\pi),\Dinv(\pi)$ and $\Skips(\pi)$.  Define 
\begin{equation} \Area(\pi)=\{(x,y):(x,y)\text{ is below path } \pi \text{ and contains a positive rank}\}. \end{equation}
For the cells above the path define
\begin{equation} \label{dinv} \Dinv(\pi)=\left\{(x,y)\text{ above the path}:\frac{\arm(x,y)}{\leg(x,y)+1}<\frac{m}{n}<\frac{\arm(x,y)+1}{\leg(x,y)}\right\}, \end{equation}
where $\arm(x,y)$ is the number of cells above $\pi$ and strictly east of $(x,y)$ and $\leg(x,y)$ is the number of cells above $\pi$ and strictly south of $(x,y)$.  We interpret division by zero to be infinity.  Define the $\Skips$ set to be the remaining cells above the path:
\begin{equation} \Skips(\pi)=\{(x,y)\text{ above the path}:(x,y)\not\in\Dinv(\pi) \}\;. \end{equation}
Set
\[ \area(\pi)=|\Area(\pi)|,\qquad \dinv(\pi)=|\Dinv(\pi)|, \qquad \skips(\pi)=|\Skips(\pi)|. \]
The idea is that when counting the cells that contribute to $\dinv$ we should \emph{skip} any that do not satisfy the inequality in~\eqref{dinv}.

Presented in \cite{Kalis15,Kalis16}, the \emph{fast $\dinv$} algorithm gives a faster and simpler way to determine whether or not a cell is in $\Dinv$.  For a given cell $(x,y)$ above an $(m,n)$-Dyck path $\pi$, let $(x,y)^\downarrow$ be the southernmost cell in the same column as $(x,y)$ that is above the path and $(x,y)^\Downarrow$ be the northermost cell in the same column as $(x,y)$ that is below the path.  Similarly, let $(x,y)^\rightarrow$ be the easternmost cell in the same row as $(x,y)$ that is above the path and $(x,y)^\Rightarrow$ be the westernmost cell in the same row as $(x,y)$ that is below the path.  So $(x,y)^\Rightarrow$ is exactly one cell to the east of $(x,y)^\rightarrow$ and $(x,y)^\Downarrow$ is exactly one cell south of $(x,y)^\downarrow$.

\begin{theorem}[Fast $\dinv$] \label{fastdinvthm}
Suppose $\pi$ is an $(m,n)$-Dyck path and let $(x,y)$ be a cell above the path in $\pi$.  The cell $(x,y)$ is in $\Dinv(\pi)$ if and only if
\begin{equation} \label{fastdinv} \gamma[(x,y)^\rightarrow]>\gamma[(x,y)^\Downarrow] \qquad \text{and}\qquad \gamma[(x,y)^\downarrow]>\gamma[(x,y)^\Rightarrow]. \end{equation}
\end{theorem}

\begin{example} Consider cell $(1,7)$ of the $(5,7)$-Dyck path:
\[ \pi=\dyckpath[(1,7)]{5}{7}{0,0,1,1,1,1,2}{3cm}. \]
We compute and get
$ \gamma[(1,7)^\rightarrow]=16>\gamma[(1,7)^\Downarrow]=-2\;: $
\[ \pi=\dyckpath[(2,7),(1,2)]{5}{7}{0,0,1,1,1,1,2}{3cm}. \]
But $ \gamma[(1,7)^\downarrow]=3<\gamma[(1,7)^\Rightarrow]=9\;: $
\[ \pi=\dyckpath[(3,7),(1,3)]{5}{7}{0,0,1,1,1,1,2}{3cm}. \]
So $(1,7)\not\in\Dinv(\pi)$.
\end{example}

Define the rational Catalan polynomial
\[ C_{m,n}(q,t)=\sum_{\pi} q^{\dinv(\pi)}t^{\area(\pi)} \]
where the sum is over all $(m,n)$-Dyck paths.

\section{Modified Catalan Polynomials, $K_{m,3}(q,t)$ }

For $m$ a positive integer, define
\[ K_{m,3}(q,t)=\sum_{0\leq i<m/3} s_{m-1-2i,i}(q,t). \]
where $s$ is the Schur basis for symmetric functions.  When $m$ is not divisible by 3
\[ K_{m,3}(q,t)=C_{m,3}(q,t)=C_{3,m}(q,t), \]
so $K_{m,3}$ are the generating functions for rational Dyck paths.

$K_{3k,3}$ can be realized as generating functions for modified $(3k,3)$-rational Dyck paths.  When a rank appears more than once in the $(3k,k)$-diagram then the one appearing further to the east is considered slightly larger and only \emph{positive} ranks can lie above the path.

\begin{example}
The Dyck path
\[ \pi=\dyckpath{6}{3}{0,1,4}{3cm} \]
is a valid modified $(6,3)$-Dyck path because the rank in cell $(4,3)$ is slightly larger than the $0$-rank in cell $(2,2)$, and is therefore positive.
\end{example}

Rather than showing a direct proof in the style of \cite{Kalis16}, we will make the observation that there is a bijection from the set of $(3k,3)$-Dyck paths and the set of $(3k+1,3)$-Dyck paths where the rank $1$ is below the path.  One can see the bijection by superimposing the $(3k+1,3)$-diagram over the $(3k,3)$-diagram and noting that 
\[ \gamma_{3k,3}(x,y) > \gamma_{3k,3}(u,v) \qquad\iff\qquad \gamma_{3k+1,3}(x,y) > \gamma_{3k+1,3}(u,v). \]
Thus the $\dinv$ and $\area$ statistics match precisely.

\begin{example}
Consider
\[ \pi=\dyckpath{9}{3}{0,2,3}{4cm} \qquad \text{versus} \qquad \pi=\dyckpath{10}{3}{0,2,3}{4cm}\;. \]
\end{example}

The set of $(3k+1,k)$-Dyck paths where rank $1$ must lie above the path are those whose first row contains $k$ or more cells above the path.  Suppose $\pi_i$ is the $(3k+1,k)$-Dyck path with exactly $k+i$ cells in the first row above the path.  Then for $\pi_0$ and any $x\leq k$,
\[ \gamma[(x,3)^\Rightarrow]=m-2 > m-3\geq \gamma[(x,3)^\downarrow]. \]
So $\dinv(\pi_0) = k$ and $\area(\pi_0)=k$.

For $\pi_i$ with $0<i\leq k$, not only is cell $(k+i,3)$ in $\Dinv(\pi_i)$, but for any $j\leq i$,
\begin{align*} \gamma[(j,3)^\Rightarrow] & =\gamma(k+i+1,3) \\
 & =\gamma(i,3)-3k-3 \\
 & <\gamma(j,3)-3k-3 \\
 & <\gamma(j,3)-3k-1 \\
 & = \gamma[(j,3)^\downarrow]. \end{align*}
So $\dinv(\pi_i)=k +2i$ and $\area(\pi_i)=k-i$.

The $(3k+1,3)$-Dyck paths where $1$ lies above the path contribute
\[ \sum_{i=0}^k q^{k+2i}t^{k-i}=\sum_{i=0}^k q^{3k-2i}t^i \]
to $C_{k+1,3}(q,t)$.
Note that in the bijection, a $(3k+1,3)$ path with rank $1$ below the path is mapped to a modified $(3k,3)$ path.  Since the rank $1$ is transformed into rank $0$, that cell no longer contributes to the area of the image.  So
\begin{align*} \frac{C_{3k+1,3}(q,t) - \sum_{i=0}^k q^{3k-2i}t^i}{t}
 & =\frac{\sum_{i=0}^k s_{3k-2i,i}(q,t) - \sum_{i=0}^k q^{3k-2i}t^i}{t} \\
 & =\sum_{0\leq i<k} s_{3k-1-2i,i}(q,t) \\
 & = K_{3k,3}(q,t).
 \end{align*}

\section{Rational Parking Functions}

Let $m$ and $n$ be coprime. 
An $(m,n)$-parking function is an $(m,n)$-Dyck path with the integers $1,\ldots, n$ written on the path so that within any particular column the entries are increasing from bottom to top.  We can think of an $(m,n)$-parking function $P$ as a function $P:[n]\rightarrow\Z$ where $P(i)$ is the rank of the cell containing $i$.  Thus we can write the ranks on the path as
\[ R(P)=[P(1), P(2),\ldots, P(n)]\;. \]

Another way of thinking of an $(m,n)$-parking function is as a pair $(\pi,\sigma)$ where $\pi$ is an $(m,n)$-Dyck path and
$\sigma\in\mathfrak S_n$ is a permutation.  If $R(\pi)=\{r_1,\ldots,r_n\}$ are the set of ranks on the path of $\pi$ written in increasing order then $\sigma$ is the permutation where
\[ r_{\sigma^{-1}(i)} = P(i) .\]
Each pair $(\pi,\sigma)$ can be uniquely represented by using inline or window notation:
\[ [ r_{\sigma^{-1}(1)}, r_{\sigma^{-1}(2)},\ldots,r_{\sigma^{-1}(n)}]\;. \]

In this language an $(m,n)$-parking function can be realized as a pair $(\pi,\sigma)$ where $\pi$ is an $(m,n)$-Dyck path and $\sigma$ is a permutation such that if $r_{\sigma^{-1}(i)}=k$ and $r_{\sigma^{-1}(j)}=k+m$ in $R(\pi)$ then $i<j$, i.e. in the window notation $k$ must appear left of $k+m$.  For a more classical handling of $(m,n)$-parking functions, see \cite{Bergeron16}.

\begin{example}  Consider the $(4,7)$-Dyck path:
\[ \pi=\dyckpath{4}{7}{0,0,0,0,0,2,2}{3cm}\;.\]

The set of $(4,7)$-parking functions includes:
\[ [-7,-3,-1,1,3,5,9], \qquad [-7,-1,3,-3,1,5,9], \qquad\text{and}\qquad [-1,-7,-3,1,5,9,3]. \]
\end{example}

\subsection{Statistics on Parking Functions}

The construction of the Hikita polynomial from parking functions is analagous to the construction of the Catalan polynomial from Dyck paths.  The first step is to extend the Dyck path statistics to the set of parking functions.

Given an $(m,n)$-parking function $P=(\pi,\sigma)$ we define the $\area$ to be the $\area$ of the underlying Dyck path
and the inverse descent set is the inverse descent set of the permutation $\sigma$, i.e.
\[ \area(P)=\area(\pi) \qquad\text{ and }\qquad \ides(P)=\ides(\sigma).  \]
If we look at the window $W$ for $P$ we can compute the inverse descents of the parking function by computing the descents of the window,
\[ \ides(P) = \des(W). \]

Let $m$-bounded inversions be pairs:
\begin{equation} \Inv(P)=\{(i,j) | i<j \text{ and } \sigma(j)<\sigma(i)<\sigma(j)+m\} \end{equation}
and $\inv(P)=|\Inv(P)|.$
Define the $\dinv$ of a parking function to be
\begin{equation} \label{dinvdef} \dinv(P)=\dinv(\pi)-\inv(P) . \end{equation}

\begin{remark} The definition of $\dinv$ for parking functions is given in \cite{Gorsky13,Bergeron16} as
\[ \dinv(P)=\dinv(\pi)+\tdinv(P)-\maxtdinv(\pi) \]
where $\tdinv$ is the number of pairs
\[ \{(i,j) | i<j \text{ and } \sigma(i)<\sigma(j)<\sigma(i)+m\} \]
and $\maxtdinv$ is the largest $\tdinv$ of all parking functions associated to the path.
It is not difficult to see that the definition given on line \eqref{dinvdef} is equivalent.
\end{remark}

This allows us to define the Hikita polynomials
\[ \mathcal{H}_{m,n}(X;q,t)=\sum_{\PF}t^{\area(\PF)}q^{\dinv(\PF)}F_{\ides(\PF)}(X) \]
where the sum is over all $(m,n)$-parking functions and $\{F_\alpha\}$ is the set of Gessel's fundamental
quasisymmetric functions.

It is immediate by the descriptions of the statistics that if $P$ is a parking function on Dyck path $\pi$ with window
\[ \sigma=[w_1,w_2,\ldots,w_n] \]
where $w_i<w_{i+1}$ for all $i$, then
\[ \dinv(P)=\dinv(\pi). \]

\subsection{Affine Permutations and Parking Functions}

There is a correspondence between $(m,n)$-parking functions and $m$-bounded affine $n$-permutations, see \cite{Gorsky14a}.  Suppose that $P$ is a parking function.  To obtain the affine permutation from the window notation of $P$, add $\kappa=n+1-\area(P)$ to each rank:
\[ \sigma=[w_1,w_2,\ldots,w_n]\longrightarrow [w_1+\kappa,w_2+\kappa,\ldots,w_n+\kappa]=\hat\sigma. \]
These affine permutations are $m$-bounded because for all $i$,
\[ \sigma^{-1}(i)<\sigma^{-1}(i+m). \]

The statistics of the parking function can be recovered from the affine permutation.  Let $a$ be the smallest letter appearing in the window of $\hat\sigma$; then
\[ \area(P)=1-a. \]
An $m$-bounded inversion is a pair $(i,j)$ such that
\[ i<j, \quad 1\leq j \leq n, \quad \text{and} \quad \hat\sigma(j) <\hat\sigma(i) < \hat\sigma(j)+m. \]
Let the number of $m$-bounded inversions of $\hat\sigma$ be denoted $\inv(\hat\sigma)$, which gives
\[ \dinv(P)=\frac{(m-1)(n-1)}{2}-\inv(\hat\sigma). \]

\subsection{The $F_{[n-1]}$-term}

Let $m$ and $n$ be coprime.  The highest order term $F_{[n-1]}$ of $\mathcal H_{m,n}$ occurs when the parking function consists of a Dyck path $\pi$ and the longest permutation $w_0\in\mathfrak S_n$.  This can only happen when for each rank $k\in R(\pi)$, $k+m$ does not appear in $R(\pi)$.  Another way of saying this is that each rank on the path $\pi$ must lie in a distinct column of the $(m,n)$-diagram.

Since this can only happen when $m>n$, we have the following proposition:
\begin{prop} The coefficient of $F_{[n-1]}$ is 0 in all $\mathcal H_{m,n}$ where $m<n$. \end{prop}

For $m>n$ let $\pi$ be a $(m,n)$-Dyck path with ranks $R(\pi)$ in distinct columns.  Note that the ranks appearing in
row $i$ of the diagram are $(i-1)m-kn$ for $k=1,\ldots, m$.  However, since each rank on the path is in a distinct column, it must be that
\[ R(\pi)_i \geq (i-1)m-in. \]

If we observe the $(m-n,n)$-rank diagram, we see that its $i\th$ row contains the ranks $(i-1)(m-n)-kn=(i-1)m-(i-k-1)n$ for $k=1,\ldots,m-n$.  Thus each rank in row $i$ must occur in the $i\th$ row of the underlying diagram for $\pi$.  Furthermore, each positive rank $\pi$ must appear in the $(m-n,n)$-rank diagram.  This creates a natural bijection between $(m-n,n)$-Dyck paths and $(m,n)$-Dyck paths whose ranks on the path lie in distinct columns.

\begin{theorem} \label{mntheorem} For any $m,n$ coprime with $m>n$, the coefficient of $F_{[n-1]}$ in $\mathcal H_{m,n}$ is $C_{m-n,n}$.
\end{theorem}

\begin{proof}  Suppose that $\pi$ is an $(m,n)$-Dyck path with $m>n$ coprime and each rank on the path in a distinct column.  The term in the $F_{[n-1]}$-coefficient in $\mathcal H_{m,n}$ arising from $\pi$ corresponds to the parking function
\[ \sigma = \{w_1,w_2,\ldots,w_n\} \]
where $w_i>w_{i+1}$ for all $i$.

This parking function corresponds to an affine permutation, which can be factored into an affine part and a standard permutation $\hat\sigma=(\eta,\omega_0)$.  We claim that $(\eta, id)$ is $(m-n)$-bounded and 
\begin{equation} \label{hardequal} \frac{(m-n-1)(n-1)}{2}-\mathsf{inv^{m-n}}(\eta,id)=\frac{(m-1)(n-1)}{2}-\inv(\eta,\omega_0). \end{equation}

To show that $(\eta,id)$ is $(m-n)$-bounded first note that the window of $(\eta,\omega_0)$ is decreasing.  So if $i$ appears in the window of $(\eta,\omega_0)$ then $i+m$ must appear right of the window.  Therefore $i+m-n$ must either appear in the window or right of the window.  But the window of $(\eta,id)$ is increasing, so $i+m-n$ must appear to the right of $i$.  Since this holds for all $i$ in the window, it holds for all $i$.

To prove Equation \eqref{hardequal} we will show that 
\[ \inv(\eta,\omega_0)-\mathsf{inv^{m-n}}(\eta,id) = \binom{n}{2}. \]
We do this by noting that for every pair $(i,j)$ in the window of $(\eta,id)$ with $i<j$, if $j>i+m$, there is a $k_j$ to the left of the window with $i+m-n<k_j<i+m$.
\end{proof}

Since for degree 2, $F_\emptyset = s_{(1,1)}$ and $F_{\{1\}}=s_{(2)}$ this gives:
\begin{cor}
For any odd $m>0$, 
\[ \mathcal H_{m,2}(X;q,t) = C_{m,2}(q,t)s_{(1,1)}(X) + C_{m-2,2}(q,t)s_{(2)}(X). \]
\end{cor}

\section{The Hikita Polynomials $\mathcal H_{m,3}$}

Theorem \ref{mntheorem} states that
\[ \mathcal H_{m,3}(q,t)= C_{m,3}(q,t)F_\emptyset(X)+p^m_1(q,t)F_{\{1\}}(X)+ p^m_2(q,t)F_{\{2\}}(X)+C_{m-3,3}(q,t)F_{\{1,2\}}(X). \]
We will continue by investigating the polynomials $p^m_1(q,t)$ and $p^m_2(q,t)$.  

Consider the set of $(m,3)$-Dyck paths for $m$ not divisible by 3.  We can uniquely describe an $(m,3)$-Dyck path by stating the number of cells above the path in rows 1 and 2, respectively.  We will denote this by $D_m(k,\ell)$, where $0\leq k <2m/3$, $0\leq \ell<m/3$ and $\ell\leq k$.

\begin{example}
\[ D_5(3,1)=\dyckpath{5}{3}{0,1,3}{4cm}\;. \]
\end{example}

For a fixed $m$ we will partition the set of $(m,3)$-Dyck paths into various types, based on $k$ and $\ell$:
\begin{itemize}
\item \underline{Type 0:}  $\{D_m(0,0)\}$,
\item \underline{Type 1:}  $\{D_m(k,k): k<m/3\}$,
\item \underline{Type 2a:}  $\{D_m(k,0):k<m/3\}$,
\item \underline{Type 2b:}  $\{D_m(k,0):k> m/3\}$,
\item \underline{Type 3a:}  $\{D_m(k,\ell): 0<\ell<k, k<m/3\},$
\item \underline{Type 3b:}  $\{D_m(k,\ell): 0<\ell, k> m/3\}.$
\end{itemize}
Define the polynomials 
\[ P^m_y(q,t) = \sum_{ \pi \text{ type $y$}} t^{\area(\pi)}q^{\dinv(\pi)}, \]
so that
\[ C_{m,3}(q,t) = \sum_{y\in\{0,1,2a,2b,3a,3b\}} P^m_y(q,t). \]
Since there is only one Dyck path of type 0,
\[ P^m_0(q,t)=t^{m-1}.  \]

\subsection{The polynomial $P^m_1$}

Let $m$ not be divisible by 3.
Let $\pi$ be a $(m,3)$-Dyck path of type $1$ and let $(x,3)$ be any cell above the path in the top row.
Consider that 
\begin{align*} \gamma[(x,3)^\Rightarrow] & = 2m-3-3k \\
 & > m - 3x \\ 
 & = \gamma(x,2) = \gamma[(x,3)^\downarrow], \end{align*}
since $x<k<m/3$.  No cells in the first row are in $\Dinv(\pi)$, and therefore
\[ P^m_1(q,t) = \sum_{1\leq k < m/3} t^{m-1-2k}q^k. \]

There are three parking functions associated to any $(m,3)$-Dyck path of type $1$:
\[ p_1=[-3,a,a+m], \qquad p_2=[a,-3,a+m], \qquad p_3=[a,a+m,-3]. \]
If $p_i$ are the parking functions associated to the Dyck path $D_m(k,k)$ then
$a=m-3-3k<m-3$.  Note that $\des(p_1)=\emptyset, \des(p_2)=\{1\}$ and $\des(p_3)=\{2\}$.

So $(1,2)$ is an $m$-bounded inversion in $p_2$ and $(1,3)$ is an $m$-bounded inversion in $p_3$.  Therefore 
\begin{equation} \label{pfdinv1} \dinv(p_2)=\dinv(p_3)=\dinv(D_m(k,k))-1. \end{equation}

\begin{example}
The type 1 $(5,3)$-Dyck path is:
\[ D_5(1,1)=\dyckpath{5}{3}{0,1,1}{4cm}\;,\]
so
\[ P^5_1(q,t)=q^1t^2. \]
There are three associated parking functions:
\[ [-3,-1,4], \qquad [-1,-3,4], \qquad [-1,4,-3] \]
which contribute terms
\[ q^1t^2F_\emptyset, \qquad t^2F_{\{1\}}, \qquad t^2F_{\{2\}},\]
respectively.
\end{example}

\subsection{The polynomial $P^m_{2a,S}$}

Let $m$ not be divisible by 3.
Let $\pi$ be a $(m,3)$-Dyck path of type $2a$ and let $(x,3)$ be any cell above the path in the top row.
Consider that 
\begin{align*} \gamma[(x,3)^\rightarrow] & = 2m-3k \\
 & > m - 3x \\ 
 & = \gamma(x,2)=\gamma[(x,3)^\Downarrow], \end{align*}
since $x<k<m/3$.  Every cell above the path is in $\Dinv(\pi)$, and therefore
\[ P^m_{2a}(q,t) = \sum_{1\leq k< m/3} t^{m-1-k}q^k. \]

There are three parking functions associated to any $(m,3)$-Dyck path of type $2a$:
\[ p_1=[-3,m-3,b], \qquad p_2=[-3,b,m-3], \qquad p_3=[b,-3,m-3]. \]
If $p_i$ are the parking functions associated to the Dyck path $D_m(k,0)$ then
$m-3 < b=2m-3-3k<2m-3$, since $k<m/3$.  Note that $\des(p_1)=\emptyset, \des(p_2)=\{2\}$ and $\des(p_3)=\{1\}$.

So $(2,3)$ is an $m$-bounded inversion
in $p_2$ and $(1,3)$ is an $m$-bounded inversion in $p_3$.  Therefore
\begin{equation} \dinv(p_2)=\dinv(p_3)=\dinv(D_m(k,0))-1. \end{equation}

\begin{example}
The type 2a $(5,3)$-Dyck path is:
\[ D_5(1,0)=\dyckpath{5}{3}{0,0,1}{4cm}\;,\]
so
\[ P^5_{2a}(q,t)=q^1t^3. \]
There are three associated parking functions:
\[ [-3,2,4], \qquad [-3,4,2], \qquad [4,-3,2] \]
which contribute terms
\[ q^1t^3F_\emptyset, \qquad t^3F_{\{2\}}, \qquad t^3F_{\{1\}},\]
respectively.
\end{example}

\subsection{The polynomial $P^m_{2b,S}$}

Let $m$ not be divisible by 3 and let $S=\{1\}$ or $S=\{2\}$.
Let $\pi$ be a $(m,3)$-Dyck path of type $2b$ and let $(x,3)$ for $x<k-m/3$ be a cell above the path in the top row.
Consider that since $k>m/3$, 
\begin{align*} \gamma[(x,3)^\rightarrow] & = 2m-3k \\
 & < m - 3x \\ 
 & = \gamma(x,2)=\gamma[(x,3)^\Downarrow]. \end{align*}
Every cell above the path with $x<k-m/3$ is not in $\Dinv(\pi)$.  However, every cell $(x,3)$ above the path with $x>k-m/3$ is in $\Dinv(\pi)$.  This can be seen because $m-3x <2m-3k$, reversing the inequality, above.  Therefore,
\[ P^m_{2b}(q,t) = \sum_{m/3<k<2m/3} t^{m-1-k}q^{\lceil m/3 \rceil}. \]

There are three parking functions associated to any $(m,3)$-Dyck path of type $2b$:
\[ p_1=[-3,m-3,b], \qquad p_2=[-3,b,m-3], \qquad p_3=[b,-3,m-3]. \]
If $p_i$ are the parking functions associated to the Dyck path $D_m(k,0)$ then
$-3 < b=2m-3-3k<m-3$, since $k>m/3$.  Note that $\des(p_1)=\{2\}, \des(p_2)=\emptyset$ and $\des(p_3)=\{1\}$.

So $(2,3)$ is an $m$-bounded inversion
in $p_1$ and $(1,2)$ is an $m$-bounded inversion in $p_3$.  Therefore
\begin{equation} \dinv(p_1)=\dinv(p_3)=\dinv(D_m(k,0))-1. \end{equation}

\begin{example}
The type 2b $(5,3)$-Dyck paths are:
\[ 
  D_5(2,0)=\dyckpath{5}{3}{0,0,2}{4cm} 
    \qquad\text{and}\qquad 
  D_5(3,0)=\dyckpath{5}{3}{0,0,3}{4cm}\;,
\]
so
\[ P^5_{2b}(q,t)=q^2t^2+q^2t^1. \]
$D_5(2,0)$ has three associated parking functions:
\[ [-3,2,1], \qquad [-3,1,2], \qquad [1,-3,2] \]
which contribute terms
\[ q^1t^1F_{\{2\}}, \qquad q^2t^1F_\emptyset, \qquad q^1t^1F_{\{1\}},\]
respectively.
\end{example}

\subsection{The polynomial $P^m_{3a,S}$}

Let $m$ not be divisible by 3.
There are six parking functions associated to any $(m,3)$-Dyck path of type $3a$:
\[ p_1=[-3,a,b], \qquad p_2=[-3,b,a], \qquad p_3=[a,-3,b], \]
\[ p_4=[a,b,-3], \qquad p_5=[b,-3,a], \qquad p_6=[b,a,-3], \]
with $a<b$.

Let $\pi$ be a $(m,3)$-Dyck path of type $3a$ and let $(x,3)$ be a cell above the path in the
top row such that $(x,2)$ is also above the path.  Since $k<m/3$,
\begin{align*} \gamma[(x,3)^\Rightarrow] & = 2m-3-3k \\
 & > m-3x \\
 & = \gamma(x,2)=\gamma[(x,3)^\downarrow].
\end{align*}
So every cell in the second row that is above the path is immediately below a cell that is not
in $\Dinv(\pi)$.  But for any cell $(x,3)$ that does not have a cell immediately below it that
is above the path,
\begin{align*} \gamma[(x,3)^\rightarrow] & = 2m-3k \\
 & > m-3x \\
 & = \gamma(x,2)=\gamma[(x,3)^\Downarrow],
\end{align*}
so it is in $\Dinv(\pi)$.
Therefore,
\[ P^m_{3a}(q,t) = \sum_{2\leq k< m/3} \sum_{\ell=1}^{k-1} t^{m-1-k-\ell}q^k. \]

If $p_i$ are the parking functions associated to the Dyck path $D_m(k,\ell)$ then
$a = m-3\ell$ and $b=2m-3k$ with $m-3<b<a+m$ since $k<m/3$.  Note that
\[ \des(p_1)=\emptyset, \qquad \des(p_2)=\des(p_4)=\{2\}, \]
\[ \des(p_3)=\des(p_5)=\{1\},\qquad \des(p_6)=\{1,2\}. \]

So $(2,3)$ is an $m$-bounded inversion in $p_2$, $(1,2)$ is an $m$-bounded inversion
in $p_3$, and $(1,3)$ is an $m$-bounded inversion in $p_4$, and $p_5$.  Therefore
\begin{equation} \dinv(p_2)=\dinv(p_5)=\dinv(p_3)=\dinv(p_4)=\dinv(D_m(k,\ell))-1. \end{equation}

\begin{example}
The $(3,5)$-diagram does not support any Dyck paths of type $3a$.  So
\[ P^5_{3a}(q,t)=0. \]
\end{example}

\subsection{The polynomial $P^m_{3b,S}$}

Let $m$ not be divisible by 3.
Let $\pi$ be a $(m,3)$-Dyck path of type $3b$.  We will break these Dyck paths into two further subcases, those with $\ell<k-m/3$ and those with $\ell>k-m/3$.  In the first case, when $\ell<k-m/3$, if $\ell<x<k-m/3$ then the cell
$(x,3)$ is not in $\Dinv(\pi)$.   This is because
\begin{align*} \gamma[(x,3)^\Downarrow] & = m - 3x \\
 & > 2m - 3k \\
 & = \gamma[(x,3) ^ \rightarrow]. \end{align*}
The cells $(x,3)$ where $x\leq \ell$ are all in $\Dinv(\pi)$ because
\begin{align*} \gamma[(x,3)^\downarrow] & = m - 3x \\
 & > 2m -3 - 3k \\
 & = \gamma[(x,3) ^ \Rightarrow]. \end{align*}
Therefore, these $(m,3)$-Dyck paths contribute
\begin{equation} \label{part3b1} \sum_{m/3+1<k<2m/3} \left(\sum_{1\leq\ell < k-m/3} q^{\lceil m/3 \rceil + 2\ell}t^{m-1-k-\ell}\right)\;. \end{equation}

In the second case, where $\ell >k-m/3$, then every cell $(x,3)$ that is immediately above a cell in the second row that is above the path and $x>k+1-m/3$ is not in $\Dinv(\pi)$.  This is because
\begin{align*} \gamma[(x,3)^\downarrow] & = m - 3x \\
 & < 2m - 3 - 3k \\
 & = \gamma[(x,3) ^ \Rightarrow]. \end{align*}
However, for every $(x,3)$ with $x<k+1-m/3$,
\begin{align*} \gamma[(x,3)^\downarrow] & = m - 3x \\
 & > 2m - 3 - 3k \\
 & = \gamma[(x,3) ^ \Rightarrow], \end{align*}
so the cell is in $\Dinv(\pi)$.  Therefore, these $(m,3)$-Dyck paths contribute 
\begin{equation} \label{part3b2} \sum_{ m/3<k< 2m/3} \left(\sum_{ k-m/3<m/3} q^{2k-\lfloor m/3 \rfloor}t^{m-1-k-\ell}\right)\;. \end{equation}

By summing Lines \eqref{part3b1} and \eqref{part3b2},
\begin{multline*} P^m_{3b}(q,t) =
\sum_{m/3+1<k<2m/3} \left(\sum_{1\leq\ell< k-m/3} q^{\lceil m/3 \rceil + 2\ell}t^{m-1-k-\ell}\right)
\\+
\sum_{m/3<k<2m/3} \left(\sum_{k-m/3<\ell<m/3} q^{2k-\lfloor m/3 \rfloor}t^{m-1-k-\ell}\right)
\;.\end{multline*}

The type 3b $(m,3)$-Dyck paths also have six parking
functions associated to each of them:
\[ p_1=[-3,a,b], \qquad p_2=[-3,b,a], \qquad p_3=[a,-3,b], \]
\[ p_4=[a,b,-3], \qquad p_5=[b,-3,a], \qquad p_6=[b,a,-3], \]
where $a$ is the rank appearing in row $2$ and $b$ is the rank appearing in row $1$.

If $p_i$ are the parking functions associated to the Dyck path $D_m(k,\ell)$ then
$a = m-3\ell$ and $b=2m-3k$ with $b<2m-3$ and $a,b > m-3.$  If $a<b$ then the case proceeds similarly to type $3a$:
\[ \des(p_1)=\emptyset, \qquad \des(p_2)=\des(p_4)=\{2\}, \]
\[ \des(p_3)=\des(p_5)=\{1\},\qquad \des(p_6)=\{1,2\}. \]

Also, $(2,3)$ is an $m$-bounded inversion in $p_2$ and $(1,2)$ is an $m$-bounded inversion in $p_3$.  However, $(1,3)$ and $(2,3)$ are $m$-bounded inversions in $p_4$ and $(1,2)$ and $(1,3)$ are $m$-bounded inversions in $p_5$.   Therefore
\begin{align} \label{type3b1a} \dinv(p_2)=\dinv(p_3)&=\dinv(D_m(k,\ell))-1,\\
\label{type3b1b}\dinv(p_4)=\dinv(p_5)&=\dinv(D_m(k,\ell))-2. \end{align}

If $b<a$ then 
\[ \des(p_1)=\des(p_6)=\{2\}, \qquad \des(p_2)=\emptyset,  \]
\[ \des(p_4)=\{1,2\}, \qquad \des(p_3)=\des(p_5)=\{1\}. \]

In this case, $(2,3)$ is an $m$-bounded inversion of $p_1$ and $(1,2)$ is an $m$-bounded inversion of $p_5$.  Both $(1,2)$ and $(1,3)$ are $m$-bounded inversions of $p_3$ and $(1,3)$ and $(2,3)$ are $m$-bounded inversions of $p_6$.  Therefore
\begin{align} \dinv(p_1)=\dinv(p_5)&=\dinv(D_m(k,\ell))-1,\\
\label{pfdinv3b22}\dinv(p_3)=\dinv(p_6)&=\dinv(D_m(k,\ell))-2. \end{align}

\begin{example}
The type 3b $(5,3)$-Dyck paths are:
\[ 
  D_5(2,1)=\dyckpath{5}{3}{0,1,2}{4cm}
    \quad\text{and}\quad 
  D_5(3,1)=\dyckpath{5}{3}{0,1,3}{4cm}\;,
\]
so
\[ P^5_{3b}(q,t)=q^3t^1+q^4. \]
$D_5(2,1)$ has six associated parking functions:
\[ [-3,-1,1], \qquad [-3,1,-1], \qquad [-1,-3,1], 
\qquad [-1,1,-3], \qquad [1,-3,-1], \qquad [1,-1,-3] \]
which contribute terms
\[ q^3t^1F_\emptyset, \qquad q^2t^1F_{\{2\}}, \qquad q^2t^1F_{\{1\}},
   \qquad q^1t^1F_{\{2\}}, \qquad q^1t^1F_{\{1\}}, \qquad t^1F_{\{1,2\}}, \]
respectively.
\end{example}

\subsection{The polynomials $K_{m-1,3}$ and $K_{m-2,3}$}

Let $m$ not be divisible by 3.  

Lines \eqref{pfdinv1}--\eqref{pfdinv3b22} give that the coefficient of $F_{\{1\}}$ in $\mathcal H_{m,3}$ is 
\[ p^m_1=q^{-1}\cdot(P^m_1+P^m_{2a}+P^m_{2b}+2P^m_{3a})+(q^{-1}+q^{-2})\cdot P^m_{3b}. \]
We claim that 
\[ K_{m-1,3}=q^{-1}\cdot(P^m_{2a}+P^m_{2b}+P^m_{3a}+P^m_{3b})\]
and
\[ K_{m-2,3}=q^{-1}\cdot(P^m_{1}+P^m_{3a})+q^{-2}\cdot P^m_{3b}. \]

Consider $q^{-1}\cdot s_{(a,b)}(q,t)$ for $a\geq b>0$,
\begin{align*} q^{-1}\cdot s_{(a,b)}(q,t) & = q^{-1} \cdot \sum_{i=b}^a q^{a+b-i}t^i \\
 & = \sum_{i=b}^a q^{(a-1)+b-i}t^i \\
 & = q^{b-1}t^a + \sum_{i=b}^{a-1} q^{(a-1)+b-i}t^i \\
 & = q^{b-1}t^a + s_{(a-1,b)}(q,t) \end{align*}
where $s_{(b-1,b)}(q,t)=0$.  If we sum over each type, $S$,
\[ \sum_S P^m_S = C_{m,3}, \]
so
\[ p^m_1 = q^{-1}\cdot (C_{m,3}-P^m_0)+q^{-1}P^m_{3a}+q^{-2}P^m_{3b}. \]
Consider that
\begin{align*} q^{-1}\cdot (C_{m,3}-t^{m-1}) & = \frac{\sum_{0\leq i<m/3} s_{(m-1-2i,i)}(q,t)-t^{m-1}}{q} \\
 & = \frac{s_{(m-1)}(q,t)-t^{m-1}+\sum_{1\leq i<m/3} s_{(m-1-2i,i)}(q,t)}{q}  \\
 & = s_{(m-2)}(q,t)+\sum_{1\leq i<m/3} \left(s_{(m-2-2i,i)}(q,t)+q^{i-1}t^{m-1-2i}\right) \\
 & = \sum_{0\leq i<m/3} s_{(m-2-2i,i)}(q,t)+ \sum_{1\leq i<m/3}q^{i-1}t^{m-1-2i} \\
 & = K_{m-1,3}(q,t) + q^{-1}\cdot P_1^m. \end{align*}
This proves that
\[ p_1^m = K_{m-1,3} + q^{-1}(P^m_1+P^m_{3a}) + q^{-2}P^m_{3b}, \]
or
\begin{equation} \label{result1} K_{m-1,3}=q^{-1}\cdot(P^m_{2a}+P^m_{2b}+P^m_{3a}+P^m_{3b}).\end{equation}

To address the second part, we will map a $(m,3)$-Dyck path $D_m(k,\ell)$ with $\ell>0$ to the $(m-2,3)$-Dyck path $D_{m-2}(k-1,\ell-1)$.  Each cell $(x,y)$ in $D_{m-2}(k-1,\ell-1)$ has the same $\arm$ and $\leg$ as $(x+1,y)$ in $D_m(k,\ell)$ so 
\[ (x,y)\in\Dinv(D_{m-2}(k-1,\ell-1)) \qquad\iff\qquad (x+1,y)\in\Dinv(D_m(k,\ell)). \]
Similarly, it is immediate that $\area(D_{m-2}(k-1,\ell-1))=\area(D_m(k,\ell))$.

Therefore, we need to consider cells $(1,2)$ and $(1,3)$ in $D_m(k,\ell)$ to compute $\dinv(D_{m-2}(k-1,\ell-1))$.  Since $(1,2)$ is in the second row, $(1,2)\in\Dinv(D_m(k,\ell))$.  The cell $(1,3)$ is in $\Dinv(D_m(k,\ell))$ if and only if
\begin{align*} \gamma_m[(1,3)^\Rightarrow] < m-3 &\qquad\iff\qquad k>m/3 \\
 &\qquad\iff\qquad D_m(k,\ell) \text{ is type 3b.} \end{align*}
So if $D_m(k,\ell)$ is type 3b,
\[ \dinv(D_{m-2}(k-1,\ell-1)) = \dinv(D_m(k,\ell))-2, \]
otherwise
\[ \dinv(D_{m-2}(k-1,\ell-1)) = \dinv(D_m(k,\ell))-1. \]

Since every $(m-2,3)$-Dyck path is the image of a $(m,3)$ dyck path of type 1, 3a, or 3b by removing one cell from row 2 and one cell from row 3,
\begin{equation} \label{result2} K_{m-2,3} = q^{-1}(P^m_1+P^m_{3a})+q^{-2}P^m_{3b}. \end{equation}

\section{Symmetry Results}

In the previous section we were studying the Hikita polynomials
\[ \mathcal H_{m,3}(X;q,t)= C_{m,3}(q,t)F_\emptyset(X)+p^m_1(q,t)F_{\{1\}}(X)+ p^m_2(q,t)F_{\{2\}}(X)+C_{m-3,3}(q,t)F_{\{1,2\}}(X). \]
By observing equations \eqref{pfdinv1}---\eqref{pfdinv3b22} we note that $p^m_1=p^m_2$, for all $m$.  This means that $H_{m,3}$ is symmetric in $X$ and can be written
\[ \mathcal H_{m,3}(X;q,t)= C_{m,3}(q,t)s_{(1,1,1)}(X)+p^m(q,t)s_{(2,1)}(X)+C_{m-3,3}(q,t)s_{(3)}(X), \]
where $p^m(q,t)=p^m_1(q,t)=p^m_2(q,t)$.

In light of Equations \eqref{result1} and \eqref{result2}, we can express $p^m$ as
\[ p^m(q,t)=K_{m-1,3}(q,t) + K_{m-2,3}(q,t). \]
So
\[ \mathcal H_{m,3}(X;q,t)= C_{m,3}(q,t)s_{(1,1,1)}(X)+(K_{m-1,3}(q,t) + K_{m-2,3}(q,t))s_{(2,1)}(X)+C_{m-3,3}(q,t)s_{(3)}(X). \]

Each of these (modified) rational Catalan polynomials is symmetric in $q$ and $t$, \cite{gorsky14,Kalis15}.  That is,
\[ C_{a,3}(q,t)=C_{a,3}(t,q) \]
and
\[ K_{b,3}(q,t)=K_{b,3}(t,q) \]
for any $b$ and any $a$ not divisible by 3.
Therefore the Hikita polynomials $\mathcal H_{m,3}$ are symmetric in $q$ and $t$:
\[ \mathcal H_{m,3}(X;q,t)=\mathcal H_{m,3}(X;t,q). \]

\section{Acknowledgements}
The authors would like to thank Adriano Garsia and Jim Haglund for the motivation to work
on this problem.  We would also like to thank Angela Hicks, Huilan Li, and Emily Sergel for many valuable discussions on 
Dyck paths and Hikita polynomials.

\bibliographystyle{alpha}
\bibliography{catalanbib}

\end{document}